\documentclass{birkmult}
 \usepackage{amsfonts}
\usepackage{amssymb}
\usepackage{amsthm}
\usepackage{amsmath}
\makeatletter
\@addtoreset{equation}{section}
\makeatother
\numberwithin{equation}{section}

\newtheorem{Thm}{Theorem}[section]
\newtheorem{lemma}{Lemma}[section]
\newtheorem{prop}{Proposition}[section]
\newtheorem{cor}{Corollary}[section]
\newtheorem{Def}{Definition}[section]
\newtheorem{Rem}{Remark}
\newtheorem{Conj}{Conjecture}
\newtheorem{Clm}{Claim}

\renewcommand{\theequation}{\thesection.\arabic{equation}}

\newcommand{\1}{{\bold 1}}

\newcommand{\C}{\mathbb C}
\newcommand{\R}{\mathbb R}

\renewcommand{\(}{\left(}
\renewcommand{\)}{\right)}

\renewcommand{\a}{\alpha}

\renewcommand{\phi}{\varphi}
\newcommand{\eps}{\varepsilon}

\renewcommand{\le}{\leqslant}
\renewcommand{\ge}{\geqslant}

\begin{document}
%
%
%
%
%
%
%
%
%
\title[Borcea's variance conjectures]
 {Borcea's variance conjectures on the critical points of polynomials}
\author[Khavinson]{Dmitry Khavinson}

\address{%
Department of Mathematics and Statistics\\
University of South Florida\\
4202 East Fowler Ave, PHY114\\
Tampa, FL 33620-5700\\
 USA}

\email{dkhavins@usf.edu}

\author[Pereira]{Rajesh Pereira}

\address{%
Department of Mathematics and Statistics\\
University of Guelph\\
50 Stone Road East\\
Guelph, Ontario, N1G 2W1\\
CANADA}

\email{pereirar@uoguelph.ca}

\author[Putinar]{Mihai Putinar}
\address{%
Department of Mathematics\\
University of California\\
Santa Barbara, CA 93106\\
USA}

\email{mputinar@math.ucsb.edu}

\author[Saff]{Edward B. Saff}
\address{%
Center for Constructive Approximation\\
Department of Mathematics\\
Vanderbilt University\\
Nashville, TN 37240\\
USA}

\email{Ed.Saff@Vanderbilt.Edu}

\author[Shimorin]{Serguei Shimorin}

\address{%
Department of Mathematics\\ 
Royal Institute of Technology\\            
100 44 Stockholm\\
 SWEDEN }
 
 \email{shimorin@math.kth.se}

\thanks{This work was completed with the support of the
American Institute of Mathematics, Palo Alto. D. Khavinson, M. Putinar and E. Saff also gratefully acknowledge the
support by the National Science Foundation, through the grants DMS-0855597, DMS-1001071 and DMS-0808093, respectively. R. Pereira thanks for support the Natural Sciences and Engineering Research Council of Canada, discovery grant 400096.}



\date{today}
\dedicatory{Julius Borcea, in memoriam}

\begin{abstract} Closely following recent ideas of J. Borcea, we discuss various modifications and relaxations of Sendov's conjecture
about the location of critical points of a polynomial with complex coefficients.
The resulting open problems are formulated in terms of matrix theory, mathematical statistics or potential theory. Quite a few
links between classical works in the geometry of polynomials and recent advances in the location of spectra of small
rank perturbations of structured matrices are established. A couple of simple examples provide natural and sometimes sharp
bounds for the proposed conjectures.

\end{abstract}

\maketitle

\section{Introduction}
In 1958, while working as an assistant to Professor N. Obreshkov, Blagovest Sendov raised the following question.  {\it Assume that all zeros of a polynomial $F$
with complex coefficients lie in the closed unit disk. Is it true that there exists a critical point in
every disk of radius one centered at a zero of $F$?} The conjecture was included in the 1967 book of
research problems collected and edited by W. Hayman \cite{Hayman}. The conjecture was wrongly attributed to 
Ilieff, who in fact only had communicated the problem to the world of mathematicians.

Sendov's conjecture naturally fits into the line of classical works on the geometry of polynomials
and at the same time it has close connections with potential theory and operator theory, allowing non-trivial
matrix theory reformulations. In spite of many ingenious ideas originating in attempts to solve
Sendov's conjecture,  only partial results are known today, see for instance the respective chapters in the monographs
\cite{Marden,RS,Sh}. It was the young Julius Borcea who pushed the positive solution to this problem up to
degree seven \cite{Bo1,Bo2}, and it was the mature Julius Borcea who outlined a series of extensions and modifications
of Sendov's conjecture in the context of statistics of point charges in the plane. The authors of the
present note had the privilege to work closely with Borcea on this topic, individually, and on two occasions
in full formation (generously supported by the Mathematics Institute at the Banff Center in Canada,
and by the American Institute of Mathematics in Palo Alto). Without aiming at completeness, the following pages offer a glimpse at 
Borcea's beautiful new ideas.

To understand the present status of Sendov's conjecture we outline first the main contributions scattered through
its half century history.  In the late 1960s soon after the publication of Hayman's book, a series of papers by several authors proved Sendov's conjecture for polynomials of degree three, four and five in quick succession.  Sendov's conjecture for polynomials of degree six turned out to be significantly more difficult; its solution in Julius Borcea's first research paper appeared more than a quarter century later in 1996 \cite{Bo1}.  In his next paper published in the same year, he proved that the conjecture is true for polynomials of degree seven \cite{Bo2}.   It is now known that Sendov's conjecture holds for polynomials with at most eight distinct zeros \cite{BX}.  In the twenty-five years between the proofs for polynomials of degree five and those of degree six, some other special cases of Sendov's conjecture of were solved; these tended to be classes of polynomials defined by conditions on the geometry of their zeros rather than on their degree. For instance, Sendov's conjecture is true in the following cases:

-  all zeros are real \cite{RS}, 

- for  all trinomials \cite{Sch0},

-  all zeros belong to the unit circle \cite{GRR},

-  all zeros are contained in the closed unit disk, and at least one lies at the center of the disk \cite{BRS}, 

-  if the convex hull of the zeros form a triangle \cite{Sch}. \\

The contents are the following.  We introduce Borcea's family of conjectures in the next section.  We verify the $p=2$ version of Borcea's conjecture for polynomials with at most three distinct roots in section three and the even stronger $p=1$ version of Borcea's conjecture for cubic polynomials in section four. The connection between these family of conjectures and the theory of univalent functions is explored in section five.  The next three sections detail various operator theoretical approaches to these conjectures and other results in the analytic theory of polynomials.  Using the logarithmic derivative, it is easy to see that the zeros of the derivative of a polynomial can be viewed as the zeros of the Cauchy transform of a sum of point measures at the zeros of the polynomial.  In section nine, we generalize Borcea's conjectures to the zeros of Cauchy transform of weighted point measures.  Our concluding section is somewhat more speculative; we look at some very interesting ideas in the mathematical literature which could prove useful for tackling these questions.

\section{The p-variance conjectures} Since we are interested in the geometry of zeros of a polynomial, it is natural to introduce some
quantitative estimates of their relative position in the plane. We denote by $V(F)=\{ z\in \mathbb{C}: F(z)=0\}$ the zero set of a polynomial $F \in \C[z]$.

\begin{Def} Let $F(z)=a\prod_{k=1}^{n}(z-z_k)$ be a polynomial and let $p\in (0,\infty) $.  Define the $p$-variance
of the zero set of $F$ by $\sigma_{p} (F)=\min_{c\in \mathbb{C} }(\frac{1}{n} \sum_{k=1}^{n} \vert z_k -c\vert^p )^{\frac{1}{p}}$. \end{Def}

Certain values of $\sigma_p(F)$ will be of particular importance to us: $\sigma_2(F)$ is the {\it variance} and  $\sigma_1(F)$ is called the {\it mean deviation}. We define 
$\sigma_{\infty}(F)$ as:
$$ \sigma_\infty(F) = \min_c \max_{1 \leq k \leq n} \vert z_k -c\vert$$
and we identify it as the {\it Chebyshev radius} of the set $V(F)$; the optimal value of $c$ is known as
the {\it Chebyshev center} of the set $V(F)$, see for instance \cite{BV} or ample references in the Banach space literature \cite{Rao}. 

\begin{lemma} For every polynomial $F$ and $0<p<q<\infty$ we have
$$ \sigma_p(F) \leq \sigma_q(F) \leq \sigma_\infty(F).$$
\end{lemma}

\begin{proof} Let $c$ be an optimal center for $\sigma_q$:
$$ \sigma_q(F) = (\frac{1}{n} \sum_{k=1}^{n} \vert z_k -c\vert^q )^{\frac{1}{q}}.$$
In virtue of H\"older's inequality:
$$ \sum_k |z_k-c|^p \leq [\sum_k |z_k-c|^{p \cdot q/p}]^{p/q} n^{1-p/q},$$
hence
$$\sigma_p(F) \leq [ \frac{1}{n}  \sum_k |z_k-c|^p]^{1/p} \leq [ \frac{1}{n}  \sum_k |z_k-c|^q]^{1/q} = \sigma_q(F)$$
and the first inequality follows. 

The second inequality is obtained similarly, starting with an optimal
center $c$ with respect to $\sigma_\infty$.
\end{proof}

Recall that in the case of variance $\sigma_2$ the optimal center is the {\it barycenter} of the zeros of $F$
$$ E(F) = \frac{1}{n} \sum_{k=1}^n z_k.$$ Indeed, the reader can verify with basic inner space techniques
that
$$ \min_c \sum_{k=1}^n |z_k - c|^2 = \sum_{k=1}^n |z_k - E(F)|^2.$$

Note also that in the above definitions we allow multiple roots, so that
$\sigma_p(F), \ 0<p<\infty,$ really depends on the polynomial $F$ rather than its zero set $V(F)$, although in the literature 
one also finds the notation $\sigma_p(V(F))$. Only
$\sigma_\infty(F)$ is intrinsic to $V(F)$.

We define $h(F,F^{\prime})$ to be the unsymmetrized {\it Hausdorff distance} between the zero sets $V(F)$ and $V(F^{\prime})$: 
$$ h(F,F^{\prime}) := \max_{F(z)=0} \min_{F'(w)=0} |z-w|,$$
while $H(F,F^{\prime}) := \max (h(F,F^{\prime}), h(F',F))$ is the symmetrized Hausdorff distance between the  the two sets.
Alternatively, we will use the more suggestive notation
$$ H(V(F),V(F')) = H(F,F'),$$
and all related variations.

A series of conjectures, derived by Borcea from the Ilieff-Sendov original problem, are stated below.

\begin{Conj}[Borcea variance conjectures] \label{main} Let $F$ be a polynomial of degree at least two and let $p \in [1,\infty)$.  Then $h(F,F^{\prime})\le \sigma_p (F)$. \end{Conj}

When $p=\infty$, this is Sendov's original conjecture. In view of the above lemma, the validity of Borcea's conjecture for a specific value of $p$ implies
that the inequality holds for all $q \geq p$, including $q=\infty$. We will see shortly that the conjecture is false for $p<1$. The problem regarding the extremal configurations for
Sendov's Conjecture was raised
by Phelps and Rodrigues \cite{PR} as follows.

\begin{Conj} \label{sec} If $p>1$, then equality in the above conjecture occurs if and only if $F$ is of the form $a(z-c)^{n}-b$ for some complex numbers $a,b,c,$ whenever $\deg F \leq n$. \end{Conj}

\noindent {\bf Example 1.} Let $n>2$ and $F(z) = z (z-1)^{n-1}$. Then $F'(z) = (z-1)^{n-2}(nz-1)$, so that
$V(F) = \{ 0,1\}$ and $V(F') = \{ \frac{1}{n}, 1\}.$ We derive from here that $$h(F,F') = \frac{1}{n},\ \ \ 
 \sigma_\infty(F) = \frac{1}{2},$$
and
$$ \sigma_p(F) = \frac{1}{n^{1/p}} \min_c [ |c|^p + (n-1)|1-c|^p]^{1/p} = \frac{1}{n^{1/p}} \min_{c \in [0,1]} [ c^p + (n-1)(1-c)^p]^{1/p}.$$

If $p=1$, then $\sigma_1(F) = \frac{1}{n} \min_{c \in [0,1]} [n-1 -c(n-2)] = \frac{1}{n}$, therefore Conjecture \ref{sec} does not hold for $p=1$,
while $\sigma_p(F) \geq \sigma_1(F) = h(F,F')$ for all $p \geq 1$.

Assume that $p<1$. The only critical point in the interval $c \in [0,1]$ of the function $\phi(c) = c^p + (n-1)(1-c)^p$ is $c_0 = \frac{1}{1+ (n-1)^{1/(p-1)}}$,
with $\phi(c)$ increasing on the interval $[0,c_0]$ and decreasing on $[c_0,1]$. Thus
$$ \sigma_p(F) = \frac{1}{n^{1/p}} < h(F,F'),$$
showing that Conjecture \ref{main} cannot be extended to the case $p<1$.\\

Another question which may be of interest is whether asymptotic versions of conjecture \ref{main} can be proven. For instance:

\begin{Conj} \label{asymp} For all $n\ge 2, p \geq 1$, there exists $C_{n,p}\ge 1$ with $\lim_{n\to \infty} C_{n,p}=1$ such that $h(F,F^{\prime})\le C_{n,p} \sigma_p (F)$ provided $deg (F) \leq n$. \end{Conj}

 The case $p=\infty$ was validated in \cite{BRS} by the value $C_{n,\infty} =2^{\frac{1}{n}}$.
 
 To a polynomial $F(z) = \prod_{k=1}^n (z-z_k)$ with derivative $F'(z) = n \prod_{j=1}^{n-1} (z-w_j)$ we associate the {\it discriminant}:
 $$ {\rm Discr} (F) := \prod_{k > \ell} (z_k-z_\ell)^2$$ and the {\it resultant}:
 $$ R(F,F') := \prod_{k=1}^n \prod_{j=1}^{n-1} (z_k - w_j).$$ The fundamental relation between the two is:
 $$  |{\rm Discr} (F)| = n^n |R(F,F')|.$$ The proof of this identity follows from the observation that
$F'(z_j)=\prod_{i \neq j}^n(z_j-z_i)=n\prod_{k=1}^{n-1}(z_j-w_k).$
 
 When all points are in the closed unit disk, the maximum value of the
 discriminant is attained for $z^n-1$, for which we get ${\rm Discr} (z^n-1) = n^n$ (equivalently, equally spaced points are Fekete points for the disk). Thus we immediately obtain that if all the zeros $z_k$ of the polynomial $F$ are in the unit disk, then the zeros $w_j$ of $F'$ must satisfy
 $$ \prod_{j=1}^{n-1} |F(w_j)| = \prod_{k=1}^n \prod_{j=1}^{n-1} |z_k - w_j| \leq 1.$$
 Thus we infer\footnote{ This also follows easily from the Gauss-Lucas Theorem and the geometric fact that any point of $V(F')$ that is further than  $\sigma_{\infty}(F)$ from $V(F)$  cannot be in the
convex hull of $V(F)$. See also Proposition 4.1.} that Sendov's conjecture is true for at least one zero of $F$ and a little more: there exists a zero of $F$ with the property that
 the product of all distances to the critical points of $F$ is less than or equal to $1$.
 
The rest of the present article focuses on Conjecture \ref{main} in the two most important cases $p=1$ and $p=2$.

\section{Gauss-Lucas matrices}

In order to prove the variance conjecture for polynomials with at most three distinct zeros and for polynomials with all zeros real we need the following definition.

\begin{Def} Let $F(z)=\prod_{j=1}^{k} (z-z_j)^{m_j}$ where all the $z_{j}$ are distinct. Let $w_{1}, w_{2},..., w_{k-1}$ be the zeros of $F^{\prime}$ that are not zeros of $F$, counted with multiplicities. We define the Gauss-Lucas matrix to be the $(k-1)\times k$ matrix $G$ with entries
$$g_{ij}=\frac{m_j\vert w_i -z_j\vert^{-2}}{\sum_{\alpha =1}^{k}m_{\alpha }\vert w_i -z_{\alpha }\vert^{-2}}.$$   \end{Def}

We note that the Gauss-Lucas matrix is a stochastic matrix which maps the vector of zeros of the original polynomial into the vector of zeros of its derivative: specifically, $G$ has only non-negative entries,
the sum of all elements of any row of $G$ is equal to $1$, and
$w=Gz$ (where $w=(w_1,...,w_{k-1} )^T$ and $z=(z_1,...,z_{k } )^T$ are column vectors). Indeed,
one starts with the identity
$$ 0 = \frac{\overline{F'(w_i)} }{\overline{F(w_i)}} = \sum_{j=1}^k \frac{m_j}{\overline{w}_i - \overline{z}_j} = \sum_{j=1}^k \frac{m_j (w_i-z_j)}{|w_i-z_j|^2},$$
which yields
$$ w_i \sum_{\alpha=1}^k \frac{m_\alpha}{|w_i-z_\alpha|^2} = \sum_{j=1}^k \frac{m_j}{|w_i-z_j|^2} z_j.$$
In particular we notice that every $w_i$ is a convex combination of the $z_j's$.

Apparently it was Ces\`aro who first wrote the above system of equations relating the zeros of a polynomial to its critical points \cite{Cesaro}. Thus we have proved the following classical result.

\begin{Thm}[Gauss-Lucas] The critical points of a polynomial lie in the convex hull of its zeros.
\end{Thm}

\begin{Def} Let $F(z)=\prod_{j=1}^{k} (z-z_j)^{m_j}$ where all $z_{j}$ are distinct. We define the augmented Gauss-Lucas matrix $A$ of $F$ to be the $k\times k$ matrix whose first $k-1$ rows consist of the Gauss-Lucas matrix $G$ and whose last row is defined as follows $a_{kj}:=\frac{m_j}{n}$.\end{Def}

Since every row or column vector of a unitary matrix has Euclidean norm equal to one, we have the following well known result, see for instance \cite[section 6.3]{HJ}.

\begin{lemma} Let $M=[m_{ij}]$ be any $n\times n$ matrix and define $\Phi(M)$ to be the $n\times n$ non-negative matrix whose $(i,j)$-th entry is $\vert m_{ij} \vert^2 $. If $U$ is a unitary matrix, then $\Phi(U) $ is doubly stochastic, i.e. 
 all entries of $\Phi(U) $ are non-negative and that the sums along every column or row equal $1$.\end{lemma}

We are now able to prove the key technical result in the case of three distinct zeros. The case of two distinct zeros can be reduced via a rotation to the case
of real zeros, to be discussed in full below.

\begin{lemma} \label{main2} Any polynomial $F$ with at most three distinct zeros has an augmented Gauss-Lucas matrix which is doubly stochastic.  \end{lemma}

\begin{proof} Let $z_1$, $z_2$ and $z_3$ be the zeros of the polynomial $F$ with multiplicites $m_1$,$m_2$ and $m_3$ respectively. 
Let $$v_1=(\frac{\sqrt{m_1}}{w_1-z_1}, \frac{\sqrt{m_2}}{w_1-z_2},\frac{\sqrt{m_3}}{w_1-z_3} )^{T}, \ \ v_2=(\frac{\sqrt{m_1}}{w_2-z_1}, \frac{\sqrt{m_2}}{w_2-z_2},\frac{\sqrt{m_3}}{w_2-z_3} )^{*}$$ (note that we take the transpose for $v_1$; but we take the conjugate transpose for $v_2$) and $v_3=(\sqrt{m_1},\sqrt{m_2},\sqrt{m_3})$. These three vectors are mutually orthogonal and therefore $\{\hat{v}_{i}\}_{i=1}^{3}$ is  an orthonormal basis of $\mathbb{C}^3,$ where $\hat{v}_{i}=\frac{v_i}{\Vert v_i\Vert}$. Let $U$ be the unitary matrix whose rows are $\{\hat{v}_{i}\}_{i=1}^{3}$. Then the augmented Gauss-Lucas matrix is $\Phi (U)$ which is doubly stochastic.\end{proof}

We note that the fourth degree polynomial $F(z)=z^{4}-3z^{2}-4$ has an augmented Gauss-Lucas matrix which is not doubly stochastic so the above lemma does not generalize to higher degree polynomials.  However, we are more fortunate in the case of purely real zeros.

\begin{lemma} Any polynomial $F$ all of whose zeros are real has an augmented Gauss-Lucas Matrix which is doubly stochastic.  \end{lemma}

\begin{proof} Let $\{ z_j\}_{j=1}^{m}$ be the zeros of the polynomial $F$, $m_j$ be the multiplicity of $z_j$ and let $\{ w_j\}_{j=1}^{m-1}$ be the zeros of $F^{\prime}$ which aren't also zeros of $F$. Let $v_{j}=(\frac{\sqrt{m_1}}{w_j-z_1}, \frac{\sqrt{m_2}}{w_j-z_2},...,\frac{\sqrt{m_n}}{w_j-z_m} )^{T}$ when $j\le k-1$ and let $v_{k}=(\sqrt{m_1},\sqrt{m_2},...,\sqrt{m_k})$.  Then $\{ v_j\}_{j=1}^{k}$ is an orthogonal set. Now,  the same argument as in the
proof of Lemma 3.2 yields the result.
\end{proof}

Remark that the vectors $v_{j}$ have the following operator theoretic interpretation. For simplicity we assume no multiple zeros.  Let $D={\rm diag}(z_{1},z_{2},...,z_{n})$.  Let $P$ be the projection on to the orthogonal complement of the vector $w=\frac{1}{\sqrt{n}}(1,1,...,1)$.  Let $B=PDP$.  According to \cite{Per}, the eigenvalues of $B$ are \\
$w_{1},w_{2},...,w_{n-1},0$. It can be seen that $v_j$ is an eigenvector of $B$ corresponding to the eigenvalue $w_j$ and $v_n$ is an eigenvector of $B$ corresponding to the eigenvalue $0$. If $A$ is Hermitian, then so is $B$, therefore the vectors $v_k$ are orthogonal.  If $A$ is not a linear function of a Hermitian matrix, then $B$ is not normal by a Theorem of Fan and Pall \cite{FP} (this is also proved directly in \cite{Per}).
We will return to this point of view in a subsequent section.

\begin{Thm} \label{3main} Any polynomial $F$ with at most three distinct zeros satisfies Borcea's 2-variance conjecture.   \end{Thm}

\begin{proof}  Let $z_j$ be a zero of $F$. We prove that there is a zero of $F^{\prime}$ which lies in the disk centered at $z$ of radius $\sigma_2(F)$.  If $z_j$ is a multiple zero then we are done, so let us assume that $z_j$ has multiplicity one.  Then by  Lemma \ref{main2}, $\sum_{i=1}^{2}g_{ij}=\frac{n-1}{n}$, hence there is at least one $i:1\le i\le 2$ such that $g_{ij}\ge \frac{1}{n}$. 
Let $E=E(F) = \sum_{k=1}^{3}\frac{m_k}{n} z_k$.  Then 
$$1=\frac{1}{n}\sum_{k=1}^{3}m_k \frac{w_i-z_k}{w_i-z_k}  = \frac{1}{n}\sum_{k=1}^{3}m_k \frac{E-z_k}{w_i-z_k} + \frac{1}{n}\sum_{k=1}^{3}m_k \frac{w_i-E}{w_i-z_k} =
\frac{1}{n}\sum_{k=1}^{3}m_k \frac{E-z_k}{w_i-z_k}.$$ Thus, by Cauchy-Schwarz inequality and in view of our choice of the index $i$, we obtain that:
\begin{equation} \label{GL} 1 \le (\frac{1}{n} \sum_{k=1}^3 m_k |E-z_k|^2)  (\frac{1}{n}{\sum_{k=1}^{3}m_k\vert w_i -z_k\vert^{-2}}) \le \vert w_i-z_j\vert^{-2} \sigma_{2}(F)^2.\end{equation} Retaining the extreme terms in the inequality
we obtain $|w_i-z_j|\le \sigma_{2}(F)$.  

Remark that equality would require that no entry in the $j$th column of the Gauss-Lucas matrix be strictly larger than $\frac{1}{n}$ which would force $n$ to be equal to three and all of the $z_1$, $z_2$ and $z_3$ to be equidistant from $w_i$. By examining the coefficient of the polynomial $G(z)=F(z-w_i)$  one can show that $F(z)=a(z-c)^{3}-b$ for some complex numbers $a,b,c$.\end{proof}

Furthermore, as we saw, the above Lemma implies that Borcea's variance conjecture holds
for polynomials with real zeros.

 We note that we do not need the augmented Gauss-Lucas matrix to be doubly stochastic for the above proof to work; any $n$th degree polynomial which has a Gauss-Lucas matrix in which every column has an element greater than or equal to $\frac{1}{n}$ would satisfy Borcea's 2-variance conjecture.  Unfortunately, even this weaker conjecture is not true in general as it fails for the 19th degree polynomial found in \cite{Mi}. This polynomial (after translation) is $F(z)=z^{19}-0.881444934z^{18}+0.896690269z^{17}-0.492806889$. The two columns of its Gauss-Lucas matrix corresponding to the conjugate pair $0.909090818\pm 0.330014556i$ of zeros of $F$ have all entries strictly less than $\frac{1}{19}$. 

It would be interesting to understand what other conditions besides stochasticity Gauss-Lucas matrix must satisfy. For instance,   what can be said about its column sums? What about the size of the maximal elements in each column?  

\section{The mean deviation conjecture}

We begin with a brief review of {\it apolarity}, a key concept in the geometry of polynomials, see for instance \cite{Dieu1,Marden,RS,Sh}.

\begin{Def} Let $F(z)=\sum_{k=0}^{n} {n\choose k} a_k z^k$ and $G(z)=\sum_{k=0}^{n} {n\choose k} b_k z^k$.  Then $F(z)$ and $G(z)$ are said to be apolar provided that
$\sum_{k=0}^{n} (-1)^k {n\choose k} a_k b_{n-k}=0$.  \end{Def} 

\begin{Thm}[Grace's Theorem, 1902] If $F$ and $G$ are apolar polynomials then any closed circular domain which contains all zeros of $G$ contains at least one zero of $F$. \end{Thm}

There are many equivalent forms of Grace's Theorem \cite{Grace} proved by among others Szeg\H{o} and Walsh (see section 3.4 of \cite{RS} for details).  At the time of writing, the most recent new proof of Grace's Theorem is by Borcea and Br\"and\'en and can be found in \cite{BoBr}. Grace's Theorem and apolarity are used to prove many results in the analytic theory of polynomials, including many of the known special cases of Sendov's conjecture.  A striking example of an application of apolarity to this conjecture is Borcea's necessary and sufficient conditions for Sendov's conjecture in \cite{BoNew}

We first note that we can prove a mean deviation result for the reverse direction of the unsymmetrized Hausdorff distance.

\begin{prop} Let $F$ be a polynomial of degree at least two.  Then $h(F^{\prime},F)\le \sigma_1 (F)$. \end{prop}

\begin{proof} Let $F(z)=\prod_{k=1}^{n}(z-z_k)$, where the $z_k$ are not required to be distinct. Let $w$ be a zero of $F^{\prime }$ that is not also a zero of $F$. Now pick $c$ such that $\frac{1}{n}\sum_{k=1}^{n}\vert z_k-c\vert=\sigma_1(F)$, since $\sum_{k=1}^{n}\frac{c-w}{w-z_k}=(c-w)\frac{F^{\prime}(w)}{F(w)}=0$, we have $$1=
\frac{1}{n} \sum_k \frac{c-z_k}{w-z_k}\le \sigma_1(F)\max_{1\le k\le n}\frac{1}{\vert w-z_k\vert }.$$  \end{proof}

Hence we can show for every $p\geq 1 $ that $H(F, F^{\prime})\le \sigma_p (F)$  would follow from $h(F, F^{\prime})\le \sigma_p (F)$.

The mean deviation conjecture ($p=1)$ has been verified for cubic polynomials by J. Borcea. In the remainder of this section we discuss his proof \cite{Bo6}. 

First we define the {\it circular deviation}.  Let 
$$\Omega_{n}=\{(\zeta_{1},\zeta_{2},...,\zeta_{n})\in \mathbb{C}^{n} :|\zeta_{1}|=|\zeta_{2}|=...=|\zeta_{n}|=1:\sum_{k=1}^{n}\zeta_{i}=0\}.$$  For $F(z)=a\prod_{k=1}^{n}(z-z_k)$, 
set $$\sigma_{\rm circ}(F)=\sup \{ \frac{1}{n}\vert \sum_{k=1}^{n} \zeta_k z_k \vert : (\zeta_{1},\zeta_{2},...,\zeta_{n})\in \Omega_n \}.$$

The sets $\Omega_n$ for $n=2,3,4$ can be described in simple geometric terms. Specifically,  denote $\zeta=(\zeta_1,\zeta_2,...,\zeta_n)$ and $\xi=(\xi_1,\xi_2,...,\xi_n)$; we say that $\zeta$ and $\xi$ are equivalent if there exists $c\in \mathbb{C}$ with $\vert c\vert=1$ and a permutation $\tau\in S_n$ such that $\zeta_{k}=c\xi_{\tau(k)}$. Then $\Omega_2$ is the set of all vectors equivalent to $(1,-1)$; $\Omega_3$ is the set of all vectors equivalent to $(1,\omega, \omega^2)$ where $\omega$ is a primitive cube root of unity; $\Omega_4$ is the set of all vectors equivalent to $(1,-1,c,-c)$ where $c$ is any complex number of modulus one.

It can easily be shown that $\sigma_{\rm circ}(F)\le \sigma_{1}(F)$. The inequality can be however strict. For instance, in the case of
$F(z)=z^3-z$ we have $\sigma_{\rm circ}(F)=\frac{1}{\sqrt{3}}$ and $\sigma_{1}(F)=\frac{2}{3}$.

Note that if $F$ is a cubic polynomial and $0$ is one of its zeros, one may construct a $3 \times 3$ circulant matrix $$C=\left( \begin{matrix} -a-b & a & b\\ \ b & -a-b & a\\  a & b & -a-b \end{matrix} \right)$$ whose characteristic polynomial is $F$. For the theory of circulant matrices we refer to the monograph \cite{Davis}. Then any two by two principal submatrix of $C$ has characteristic polynomial $\frac{1}{3}F^{\prime}(z)=(z+a+b)^{2}-ab$. By remarking that the circulant matrix can be diagonalized by the Fourier matrix
$$ \left( \begin{matrix} 1 & 1 & 1\\ 1 & \omega & \omega^2\\  1 & \omega^2 & \omega \end{matrix}\right), \ \ \  \omega^3 =1, \omega \neq 1,$$
one proves that $\sigma_{\rm circ}(F)=\max(\vert a \vert ,\vert b \vert )$.  Then one verifies that $F^{\prime}(z)$ is apolar to $G(z)=(z-a)(z-b)$.  Since the disk centred at zero of radius $\sigma_{\rm circ}(F)$ contains both zeros of $G$, it must contain at least one of the zeros of $F^{\prime}$.

\section{Univalence criteria}

A classical counterpart to Grace Theorem was discovered in 1915 by J. W. Alexander \cite{Alexander} and refined in 1917 by Kakeya \cite{Kakeya} (see also \cite[sec. 5.7]{Sh}), as follows.

\begin{Thm} [Alexander-Kakeya] If a polynomial $F$ of degree $n$ has no critical points in a closed disk of radius $R$, then it is univalent in the concentric disk of radius $R\sin (\frac{\pi}{n})$. \end{Thm}

The extremal polynomials for the Alexander-Kakeya Theorem (the ones for which $R\sin (\frac{\pi}{n})$ cannot be improved) are the same ones as in Conjecture \ref{sec}: $a(z-c)^n-b$.  We are
naturally led now to formulate the following conjecture.

\begin{Conj} If $F$ is an $n$th degree polynomial with $F(0)=0$, then $F$ is not univalent in any closed disk of radius larger than $\sigma_p(F)\sin (\frac{\pi }{n}), \ p \geq 1,$ centered at zero. \end{Conj}

Of course, if the conjecture is true for $p=1$, then it follows for all $p \geq 1$. By the Alexander-Kakeya Theorem, the truth of this conjecture implies the truth of Conjecture \ref{main} for the same value of $p$.
It can be easily seen that the strongest case $p=1$ of this conjecture is true for $n=2$.  In this case $F(z)=z^2-az$ and $\sigma_1(F)=|a|/2$. Since $F$ has a critical point at $a/2,$ it is clearly not univalent in any
disk larger than $|a|/2,$ and so our conjecture holds.

This \emph{non}univalence conjecture opens the possibility to use the huge body of known facts about the univalence of polynomials and analytic functions. Traditionally one normalizes the
functions by the conditions $F(0)=0$ and $F^{\prime}(0)=1$.  In our polynomial case we simply consider
 $F_n(z)=z\prod_{k=1}^{n-1}(1-z_{k}^{-1}z)=z+\cdots+a_nz^n$. If such a polynomial is univalent in the open 
 unit disk, then it is well-known, for example, that $|a_n| \leq 1/n$. For a survey of results on univalent
 polynomials, see \cite{Schmieder}.

\section{Perturbations of normal matrices}

We present below a different approach to Borcea's 2-variance conjecture  
derived from some recent matrix theory advances of Pereira \cite{Per} and Malamud \cite{Mal}.
We return to a monic polynomial
\begin{equation}
\label{v-1}
F(z)=(z-z_1)(z-z_2)\cdots (z-z_n),
\end{equation}
and the variance of its zero set $V(F) =\{z_1,\dots,z_n\}$
where we repeat the enumeration of multiple zeros.
$$\sigma_2^2(F)=\frac 1n\sum_{k=1}^n|z_k-E|^2 = \frac 1n \sum_{k=1}^n |z_k|^2-|E|^2,$$
where $E=E(F)$ is the barycenter: 
$$E=\frac 1n\sum_{k=1}^nz_k. $$

If now 
\begin{equation}
\label{v-2}
F'(z)=n(z-w_1)\dots (z-w_{n-1})
\end{equation}
and $V(F')$ is the sequence of critical points $\{w_1,\dots,w_{n-1}\}$, then recall that
the 2-variance Conjecture \ref{main} states that 
$$H(V(F), V(F'))\le \sigma_2(F),$$
where $H(\cdot, \cdot)$ is the symmetrized Hausdorff distance.

The main idea is to associate with the polynomial (\ref{v-1})
a normal operator $A$ acting in $\C^n$ so that $A$ is diagonal with entries 
$z_1,\dots,z_n$ along the diagonal. In this case $F_A(z)=F(z)$, where $F_A(z)$ is the characteristic polynomial of $A$. 
According to the terminology introduced in \cite{Per}, a unit vector $\mathbf v\in\C^n$ is called a {\it differentiator}
if 
$$\frac 1n F'(z)=F_B(z),$$
where $B=PAP^*$ and $P$ is the orthogonal projection onto the orthogonal complement
$\mathbf v^\perp$ in $\C^n$ (so that $B$ is acting in $\mathbf v^\perp$). 
A simple computation shows that any vector $\mathbf v=(v_1,\dots,v_n)\in\C^n$ satisfying 
$|v_k|=\frac 1{\sqrt n}$ is a differentiator, see for the proof the above section devoted to Gauss-Lucas matrices.
Moreover, there is an orthonormal basis of differentiators
$\mathbf v^{(0)}, \mathbf v^{(1)}, \dots, \mathbf v^{(n-1)}$, where 
$\mathbf v^{(l)}=(v^{(l)}_1,\dots,v^{(l)}_n)$ and 
\begin{equation}
\label{v-3}
v^{(l)}_k=\frac 1{\sqrt n} e^{2\pi i kl/n}. 
\end{equation}
We shall return to this basis later. 

  In the above notations, if $\mathbf v$ is a differentiator and $Q=I-P$, then $PAQ$ is a rank one 
operator 
\begin{equation}
\label{v-4}
PAQ\,:\, (x_k)_{1\le k\le n} \mapsto \(\sum_{k=1}^n x_k\overline{v_k}\)\cdot
\((z_k-E)v_k\)_{1\le k\le n}
\end{equation}
with the norm 
$$\|PAQ\|=\sigma_2(F)$$
which gives a natural operator-theoretical interpretation of the variance. 

  The conjectured inequality 
$$H(F,F')\le \sigma_2(F)$$
reduces to two independent assertions: \\

 {\bf Statement (i)}: {\it For any $z_\ell\in V(F)$, the disk $D(z_\ell,\sigma_2(F))$ (centered at $z_\ell$ 
and having radius $\sigma_2(V(F))$) contains at least one point from $V(F')$}.\\

 {\bf Statement (ii)}: {\it For any $w_k\in  V(F')$, the disk $D(w_k,\sigma_2(F))$ contains at least
one point from $V(F)$.} \\

  We discuss next a matrix theory proof of Statement (ii), already settled in the previous section by a different method. The basic observation goes back  to a paper 
  by Ptak
\cite{Ptak76}. Indeed, 
if we assume that $w=0$ is in $V(F')$ and $\mathbf x\in \mathbf v^\perp$ is such that 
$B\mathbf x=0$, then
$$0=PA\mathbf x=A\mathbf x-QA\mathbf x = A\mathbf x-QAP\mathbf x,$$
which implies 
$$\|A\mathbf x\|\le \|QAP\|\|\mathbf x\|=\sigma_2(F) \|\mathbf x\|$$
and hence $A$ must have at least one eigenvalue $z$ with $|z|\le \sigma_2(F)$ since 
$A$ is a normal operator.  

  The main difficulty in proving the 2-variance conjecture consists in proving Statement 
(i). A standard reduction to the case where $z_l=0$ transforms this question to the problem
of proving the invertibility of a normal matrix under certain spectral conditions on 
its principal submatrices. Namely, let $\mathbf v^{(0)}, \mathbf v^{(1)}, \dots, \mathbf v^{(n-1)}$
be an orthonormal basis of differentiators and $P_l$ denote the orthogonal projections to the
hyperplanes 
$\mathbf {v^{(l)}}^\perp$, and $Q_l=I-P_l$. The proof of statement (i) then reduces to the following
statement.

\begin{Conj}\label{matrix} Let $A$ be a normal matrix and let $\mathbf v^{(0)}, \mathbf v^{(1)}, \dots, \mathbf v^{(n-1)}$
be an orthonormal basis. Denote by $P_\ell$ the orthogonal projection onto
$\mathbf v^{(\ell )}$. If all eigenvalues of the compressions
 $B_\ell=P_\ell AP_\ell^*$ lie
outside the unit disk and $\|(I-P_\ell)AP_l\|\le 1$ for all $\ell$,  then
 $A$ is invertible. \end{Conj}

This question is similar to known invertibility criteria 
for diagonally dominant  matrices; this time we deal instead with codimention one principal submatrices 
rather than diagonal entries.

  The principal difficulty of this problem is that the matrices $B_l$ are not normal and the challenge is 
to transform their known spectral properties into appropriate metric properties. 
In two special cases this can be done. Namely, we will give an alternate proof of Statement (i) (and hence the variance
conjecture) in the case $n=3$. Another special case is that of polynomials with real zeros. In this case, 
a much more stronger estimate than Statement (i) is valid (see Proposition \ref{v-p-1} below). 

\subsection{A second proof of the 2-variance conjecture for $n=3$}

A critical step towards the proof is contained in the next statement.

\begin{lemma} 
\label{v-l1}
 If $B$ is an operator on $\C^2$ such that all eigenvalues $w$ of $B $ 
satisfy $|w|\ge 1$, then for any vector $\mathbf x\in\C^2$
$$\|B\mathbf x\|^2+\|B^*\mathbf x\|^2 \ge \|\mathbf x\|^2. $$
\end{lemma}
\begin{proof}
 In an appropriate orthonormal basis $B$ has the triangular form 
$$B=\(\begin{array}{cc}b_{11}& b_{12}\\ 0& b_{22}\end{array} \),$$
where 
$b_{11}$ and $b_{22}$ are eigenvalues of $B$. The lemma follows now from 
the coordinate-wise computation of the left-hand side of the desired inequality. 
\end{proof}

\begin{Rem} A similar inequality is no longer true in dimensions $n\ge 3$, even with a constant 
in front of $\|\mathbf x\|^2$ on the right hand side. It is enough to consider 
$$B=\( \begin{array}{ccc}1&-a&a^2\\0&1&-a\\0&0&1\end{array}\) $$
and 
$$\mathbf x=\(\begin{array}{c}1\\a\\1\end{array}\)$$
with a sufficiently large real $a$. \end{Rem}

 In view of the discussion prior to Conjecture \ref{matrix}, the 2-variance conjecture in the case $n=3$ is a consequence of the next result.
\begin{lemma}
\label{v-l2}
Let $A$ be a normal operator in $\C^3$ and let $\mathbf v_1, \mathbf v_2, \mathbf v_3$ be an orthonormal basis. 
For $\ell=1,2,3$, let $P_\ell$ be an orthogonal projection onto $\mathbf v_\ell^\perp$ and $Q_\ell=I-P_\ell$. Assume 
that
$$\|P_\ell AQ_\ell \|^2 < 1, \ \  \ell = 1,2,3,$$
and all operators $B_\ell=P_\ell AP_\ell^*$ acting in $\mathbf v_\ell^\perp$ have their spectra outside the
unit disk. Then $A$ is invertible. 
\end{lemma}

\begin{proof} Let $\epsilon >0$ be small enough, so that
$$ \| P_\ell A Q_\ell \|^2 \leq 1-\epsilon, \ \ \ell = 1,2,3.$$
Let $H$ denote $\C^3$ and let $H_\ell$ denote $\mathbf v_\ell^\perp$ for $\ell=1,2,3$. 
The inequality 
$$\|B_1P_1\mathbf x\|^2+\|B_1^*P_1\mathbf x\|^2+\|B_2P_2\mathbf x\|^2+\|B_2^*P_2\mathbf x\|^2+
\|B_3P_3\mathbf x\|^2+\|B_3^*P_3\mathbf x\|^2\ge $$ $$
\|P_1\mathbf x\|^2+\|P_2\mathbf x\|^2+\|P_3\mathbf x\|^2=2\|\mathbf x\|^2$$
shows that the row operator 
$$T=\(B_1,B_1^*,B_2,B_2^*,B_3,B_3^*\),$$
considered as acting from 
$H_1\oplus H_1\oplus H_2\oplus H_2\oplus H_3\oplus H_3$ into $H$ 
satisfies the estimate  $\|T^*\mathbf x\|^2\ge 2\|\mathbf x\|^2$. 

  The property that $A$ is a normal operator implies
$$\|P_\ell A^*Q_\ell\|=\|P_\ell AQ_\ell\|, \quad \ell=1,2,3.$$
Hence another row operator 
$$R=\(Q_1A,Q_1A^*,Q_2A,Q_2A^*,Q_3A,Q_3A^*\)$$
(again acting from 
$H_1\oplus H_1\oplus H_2\oplus H_2\oplus H_3\oplus H_3$  into $H$)
satisfies the estimate 
$$\|R^*\mathbf x\|^2\le (2-2\eps)\(\|Q_1\mathbf x\|^2+\|Q_2\mathbf x\|^2+\|Q_3\mathbf x\|^2\)=
(2-2\eps)\|\mathbf x\|^2. $$
This shows that $T^*+R^*$ is bounded away from zero and hence $T+R$ is surjective. 
Therefore, any $\mathbf x\in H$ has the form 
$$\mathbf x=A\mathbf y+A^*\mathbf z$$
with  appropriate $\mathbf y,\mathbf z\in H$ which implies that $A$ is invertible since
$A$ is a normal operator. 
\end{proof}

\subsection{Polynomials with real zeros}

In the case of polynomials whose all zeros are real, a much stronger result, recently proved by Borcea, is valid. 

\begin{Thm}\cite{Bo6}
\label{v-p-1}
If all zeros $z_1,\dots,z_n$ of polynomial (\ref{v-1}) are real, then for any 
$z_l\in V(F)$ the disk $D\(z_l,\frac{\sigma_2(F)}{\sqrt{n-1}}\)$ contains at least one
zero of the derivative $F'$. 
\end{Thm}

\begin{proof}
As before, we may assume that $z_\ell=0$ and that $\sigma_2^2\le 1-\eps$. 
 We keep also notations from preceding proof. 
Assuming that all operators $B_\ell$ do not have eigenvalues in the disk 
$D\(0,\frac 1{\sqrt{n-1}}\)$, we shall see that $A$ is invertible which gives a contradiction.

  Since all $z_1,\dots,z_n$ are real, $A$ is self-adjoint and hence all $B_\ell$ are also self-adjoint. 
Therefore, 
$$\|B_\ell \mathbf x\|\ge\frac 1{\sqrt{n-1}}\|\mathbf x\|$$
for any $\mathbf x\in\mathbf v_\ell ^\perp$ which implies that 
$$\sum_{k=1}^n\|B_kP_k\mathbf x\|^2\ge \frac 1{n-1}\sum_{k=1}^n\|P_k\mathbf x\|^2=\|\mathbf x\|^2. $$
Hence the row operator 
$$T=\(B_1,B_2,\dots,B_n\)$$
considered as acting from $H_1\oplus H_2\oplus \dots\oplus H_n$ into $H$ satisfies the estimate
$$\|T^*\mathbf x\|\ge \|\mathbf x\|$$
for any $\mathbf x\in H$. On the other hand, the row operator 
$$R=\(Q_1A,\dots,Q_nA\)$$
satisfies the estimate 
$\|R^*\mathbf x\|^2\le (1-\eps)\|\mathbf x\|^2$. Hence $R^*+T^*$ is bounded away from zero 
which shows that $R+T$ is surjective and hence $A$ is invertible. 
\end{proof}

\begin{Rem} The constant $\sigma_2(F)$ in Statement (ii) cannot be improved for polynomials 
with real zeros. A counterexample is $F(z)=(z^2-1)^2$ and $w=0$. \end{Rem}

\subsection {Toeplitz matrix  reformulation of the variance conjecture}

Statement (i) (and hence the 2-variance conjecture) is equivalent to the following clear matrix theory conjecture,
originally identified by Borcea: 

\begin{Conj} \label{Toep} Assume that $n\ge 3$, $a_1,\dots,a_{n-1}\in \C$ and $a_0=-\sum_{k=1}^{n-1}a_k$. 
Then the $(n-1)\times (n-1)$ Toeplitz matrix 
$$B=\(
\begin{array}{cccc}
a_0    & a_1    & \dots  & a_{n-2} \\
a_{n-1}& a_0    & a_1    & \dots   \\
\vdots & \ddots & \ddots &\vdots   \\
a_2    &\dots   &a_{n-1}  & a_0 
\end{array} \) $$
has at least one eigenvalue $\lambda$ satisfying 
$$|\lambda|^2\le \sum_{k=1}^{n-1}|a_k|^2. $$
\end{Conj}

Indeed, the operator $A$ in the orthonormal basis of differentiators 
$\mathbf V=\{ \mathbf v^{(0)}, \mathbf v^{(1)}, \dots, \mathbf v^{(n-1)}\} $
given by (\ref{v-3}) has the circulant matrix 
$$A_{\mathbf V}=\(
\begin{array}{cccc}
a_0    & a_1    & \dots  & a_{n-1} \\
a_{n-1}& a_0    & a_1    & \dots   \\
\vdots & \ddots & \ddots &\vdots   \\
a_1    &a_2     &  \dots & a_0 
\end{array} \), $$
where 
$$a_j=\frac 1{\sqrt n}\sum_{k=1}^{n}z_ke^{-2\pi i kj/n}$$
so that the sequence $(a_0,\dots,a_{n-1})$ is the discrete Fourier transform 
of the sequence $(z_1,\dots,z_n)$. Assuming that $z_n=0$ we find the dependence
$a_0=-\sum_{k=1}^{n-1}a_k$. Since $\sigma_2^2=\|P_0AQ_0\|^2$, we see that 
$$\sigma_2^2=\sum_{k=1}^{n-1}|a_k|^2. $$
Finally the matrix $B$ from Conjecture \ref{Toep} is exactly the matrix of the operator 
$P_0AP_0^*$ so that its eigenvalues are zeros of $p'$. 

\section{Numerical range methods}

We recall the definition of the numerical range of a square matrix.  Chapter 1 of \cite{HJ2} is a standard reference for this subject.

\begin{Def}  Let $A$ be an $n \times n$ matrix. The numerical range of $A$ is the set $W(A)=\{ x^*Ax: x\in \mathbb{C}^{n}, \Vert x\Vert_{2}=1 \}$.  \end{Def}

It can easily be seen that $W(A)$ is a compact subset of the complex plane which contains the spectrum of $A$ and that if $B$ is a compression of $A$ then $W(B)\subseteq W(A)$.   Toeplitz and Hausdorff independently have shown that $W(A)$ is always convex. If $A$ is normal, then $W(A)$ is the convex hull of the spectrum of $A$.  If $A$ is a $2$ by $2$ matrix then $W(A)$ is a (possibly degenerate) ellipse whose foci are the eigenvalues of $A$.  To demonstrate the usefulness of numerical ranges in the analytic theory of polynomials, we will use them to give a short proof of Marden's Theorem.  We recall that the Steiner inellipse of a triangle is the inscribed ellipse which  touches each side of the triangle at its midpoint.

\begin{prop} Let $F(z)$ be a polynomial with three distinct zeros $z_1,z_2,z_3$. Then the critical points of $F(z)$ are the foci of the Steiner inellipse of the triangle $\Delta z_1z_2z_3$.  \end{prop}

\begin{proof}
Let $D$ be a three by three diagonal matrix with diagonal entries $z_1,z_2,z_3$; then $W(D)=\Delta z_1z_2z_3$. Let $v=\frac{1}{\sqrt{3}}(1,1,1)$ and let $S=\{ x^*Dx: x\in \mathbb{C}^{3}, \langle x,v\rangle=0,  \Vert x\Vert_{2}=1 \}$.  Since $S$ is the numerical range of a two by two compression of $D$, $S$ is an ellipse contained in $\Delta z_1z_2z_3$.  Let $w=\frac{1}{\sqrt{2}}(1,-1,0)$. Since $\langle w,v \rangle =0$, $\frac{1}{2}(z_1+z_2)=w^*Dw\in S$.  Similarly, $\frac{1}{2}(z_1+z_3), \frac{1}{2}(z_2+z_3)\in S$ and $S$ is the Steiner inellipse of $\Delta z_1z_2z_3$.  Finally, as the projection onto $v^{\perp}$ is a differentiator of $D$, the foci of $S$ are the critical points of $F(z)$.
\end{proof}

The numerical range may also prove useful in attacking the following strengthening of Sendov's conjecture due to Schmeisser \cite{Sch}.

\begin{Conj}[Schmeisser]  Let $F(z)$ be an $n$th degree polynomial with $n \ge 2$ and let $\zeta$ be any complex number which is in the convex hull of the zeros of $F(z)$.  Then the closed disk centred at $\zeta$ with radius $\sigma_{\infty}(F)$ contains a critical point of $F$. \end{Conj}

We note that $\sigma_{\infty}(F)$ cannot be replaced by $\sigma_{p}(F)$ for any finite $p$ in the above conjecture. (If $n$ is sufficiently large so that $n^{-1}+2n^{-\frac{1}{p}}<1$, $F(z) = z (z-1)^{n-1}$ will be a counterexample).

Borcea has given a matrix theoretical generalization of Schmeisser's conjecture. Before stating it, we introduce the following notation. For an $n \times n$ matrix $A$, we denote by $\Sigma (A)$ its spectrum and by $A_k$  the $(n-1) \times (n-1)$ matrix obtained by removing the $k$th row and $k$th column of $A$. 

\begin{Conj} Let $A$ be an $n \times n$ normal matrix. Then $H(W(A),\bigcup_{k=1}^{n}\Sigma (A_k))\le \min_{c\in \mathbb{C}}\Vert A-cI\Vert$.  
\end{Conj}

Recall that for any $n \times n$ normal matrix $A$ there is a set of differentiators of $A$ which form an orthonormal basis of $\mathbb{C}^n$. If we only consider normal matrices for which every element of the standard basis is a differentiator then we recover Schmeisser's conjecture.

\section{Exclusion regions for the critical points}  So far we were interested in the shortest distance
from a zero of a polynomial to its critical points. In the other direction quite a few results are known about exclusion of 
critical points from regions around the zeros of a polynomial.  An ingenious combination of algebraic and geometric observations led
J. von Sz. Nagy to the following result.

\begin{prop} Let $z_1$ be a zero of multiplicity $m_1$ of the polynomial $F(z)$ of degree $n$. Assume that the equation 
$F(z) = 0$ admits at most $s$ zeros on every side of a line passing through $z_1$ and let $K$ be a closed disk
passing through $z_1$ which does not contain any other zero of $F$. Then $F'(z)\neq 0$ on the closed disk  internally tangent 
to $K$, of diameter $\frac{m_1}{m_1+s} diam(K)$.
\end{prop}

Without reproducing here the complete proof (contained in \cite{SzN2}), we discuss only a beautiful geometric
observation appearing in Sz. Nagy paper. Specifically, assume that $w_\ell$ is a critical point of $F(z) =
\prod_{k=1}^n (z-z_k),$ where the zeros $z_1,\ldots ,z_n$ are not necessarily distinct. If $w_\ell$ is not a zero of $F$, then
$$- \frac{F'(w_\ell)}{F(w_\ell)} =  \sum_{k=1}^n \frac{1}{z_k-w_\ell} = 0.$$
Choose a real line $L$ passing through $w_\ell$ of slope $e^{i\psi}$ with respect to the positive real semi-axis, and 
write the polar decompositions of the denominators above with respect to the origin $w_\ell$ and $L$ as zero-th direction:
$$ z_k - w_\ell = r_k e^{i(\phi_k + \psi)}.$$ By taking imaginary parts in the above identity we find
$$ \sum_{k=1}^n  \frac{\sin \phi_k}{r_k} = 0.$$

The diameter of the circle passing
through $z_k$ and which is tangent to $L$ at $w_\ell$ is
$$ d_k = | \frac{r_k}{\sin \phi_k}|,$$
not excluding the value $d_k = \infty$ if the point $z_k$ lies on $L$ (that is $\sin \phi_k =0$).
Let's rearrange the zeros $z_k$ so that $z_1,...,z_q$ belong to one side of $L$ and $z_{q+1},...,z_n$
belong to the other side, so that
$$ 0<d_1 \leq d_2 \leq \cdots \leq d_q,\ \ \ \ \ 0< d_{q+1} \leq d_{q+2} \leq \cdots \leq d_n.$$
Then we obtain the identity
$$ \frac{1}{d_1} +\cdots + \frac{1}{d_q} = \frac{1}{d_{q+1}}+\cdots+\frac{1}{d_n}.$$
From here Sz. Nagy obtains a series of non-trivial quantitative relations, of the form
$$ \frac{m_1}{d_1} \leq \frac{n-q}{d_{q+1}}, \ \ \frac{m_{q+1}}{d_{q+1}} \leq \frac{q}{d_1},$$
where $m_j$ is the multiplicity of the zero $z_j$.

A consequence of the above proposition (and the preceding geometric reasoning) is the following result,
independently obtained by Alexander \cite{Alexander} and Walsh \cite{Walsh1}.

\begin{Thm}[Alexander-Walsh] \label{AW}Let $z_1$ be a  zero of multiplicity $m_1$ of the polynomial $F(z)$ of 
degree $n$ and let $d$ denote the shortest distance from $z_1$ to another zero of $F$. Then
$F'(z) \neq 0$ for $z$ belonging to the open disk centered at $z_1$ of radius $m_1 \frac{d}{n}$.
\end{Thm} 

We offer below an alternate matrix theory proof. For any $m,p \in \mathbb{N}$ with $m\le p$, we define
$\mathcal{Q}_{m,p}$ to be the set of all $m$-tuples of integers $\beta =(\beta_1 , \beta_2 ,...,\beta_m)$
satisfying $1\le \beta_1 < \beta_2 <\ldots < \beta_{m-1} <\beta_m \le p$. If $A$ is a $p \times p$ matrix and $\beta \in \mathcal{Q}_{m,p}$,
then $\vert A[\beta ]\vert $ denotes the determinant of  the $m \times m$ principal submatrix of $A$ whose $(i,j)$th entry is the $(\beta_i ,\beta_j )$th
entry of $A$. We need the following well-known result:

\begin{lemma} \cite{HJ} Let $A$ be a $p \times p$ matrix and let $\sum_{k=0}^{n}a_{k}z^{k}$
be the characteristic polynomial of $A$. Then $a_{p}=1$ and
$a_{k}=(-1)^{p-k}\sum_{\beta \in \mathcal{Q}_{p-k,p}} \vert A[\beta ]\vert$ for $k>0$.  \end{lemma}

Let $I_k$ denote the $k \times k$ identity matrix and let $J_k$ denote the $k \times k$ matrix all of whose entries are one. Using the previous lemma and some straightforward calculations one can prove the following:

\begin{cor} Let $D$ be a $p \times p$ diagonal matrix with characteristic polynomial $G$ and let $n> 0$. Then the characteristic polynomial of $(I_{p}-\frac{1}{n}J_{p})D$ is $(1-\frac{p}{n})G+\frac{z}{n}G'$.  \end{cor}

We now prove Theorem \ref{AW}.  Let $G(z)=z^{-m_1}F(z+z_1)$. Then $G(z)$ is an $(n-m_1)$th degree polynomial with no zero in the open disk centered at the origin of radius $d$.  Since $F'(z+z_1)=z^{m_1-1}(m_1G(z)+zG'(z))$, we will be done if we show that $m_1G(z)+zG'(z)$ has no zeros in the open disk centered at the origin of radius $m_1 \frac{d}{n}$. Now let $D$ be an $(n-m_1) \times (n-m_1)$ diagonal matrix whose characteristic polynomial is $G(z)$. It is clear that $\Vert D^{-1}\Vert=\frac{1}{d}$. Let $M=I_{n-m_1}-\frac{1}{n}J_{n-m_1}$, then $M$ is a positive definite matrix with one eigenvalue equal to $\frac{m_1}{n}$ and all other eigenvalues equal to one.  Hence $\Vert M^{-1}\Vert=\frac{n}{m_1}$.  The zeros of $m_1G(z)+zG'(z)$ are the eigenvalues of $MD$. Hence if $z$ is a zero of $m_1G(z)+zG'(z)$, we have $\frac{1}{\vert z \vert}\le \Vert (MD)^{-1}\Vert \le \Vert M^{-1} \Vert \Vert D^{-1}\Vert=\frac{n}{m_1 d}$ and therefore $\vert z\vert \geq m_1 \frac{d}{n}$.\ \ \ \ \ \ \ \ \ \ \ \ \ \ \ \ \ \ \ $\square $

\section{Cauchy transforms and the weighted variance conjecture}

  The localization of zeros of the Cauchy transforms of positive measures 
is closely related to the similar problem for derivatives of polynomials. Indeed, given a polynomial 
(\ref{v-1}), we have
$$\frac{F'(z)}{F(z)}=\sum_{k=1}^n \frac 1{z-z_k}$$
and hence the zeros of $F'$ can be interpreted as zeros
of the Cauchy transform of the measure $\sum_{k=1}^n\delta_{z_k}$. 

  In the case when the support of the measure is finite, 
the problem reduces to the question of localization of zeros of the {\it Cauchy transform}:
\begin{equation}
\label{w-1}
C_\mu (z) := \sum_{k=1}^n\frac{\alpha_k}{z-z_k},
\end{equation}
where $\mu= \sum_{k=1}^n \alpha_k \delta_{z_k},$ with $ z_1,\dots,z_n \in\C$ and $\alpha_1,\dots,\alpha_n>0$. For simplicity, we may assume that 
$\sum_k \alpha_k=1$ so that $\mu$ is a probability measure. As before, we may associate with
the measure $\mu$ its baricenter 
$$E(\mu) :=\sum_{k=1}^n \alpha_k z_k$$ 
and the variances
$$\sigma_p(\mu)=\(\sum_{k=1}^n \alpha_k|z_k-E|^p\)^{1/p},  \ \  1 \leq p < \infty. $$

At this point  there are several possibilities to bring operator theory into the play. 
The first one is to consider the same normal operator $A$  as before, i.e.  the 
diagonal operator with $z_1,\dots,z_n$ on the diagonal. If the vector $\mathbf v\in\C^n$ 
is defined as $\mathbf v=\(\alpha_1^{1/2},\dots,\alpha_n^{1/2}\)$ and $P$ is the orthogonal 
projection to $\mathbf v^\perp$, then one can easily check that the eigenvalues of the operator 
$B=PAP^*$  acting in $\mathbf v^\perp$ are exactly the zeros of the Cauchy transform (\ref{w-1}).
Indeed, 
\begin{equation}
\label{wnew} \sum_{k=1}^n \frac{\alpha_k}{z_k-z} = \langle (A-z)^{-1} \mathbf v,\mathbf v \rangle = \frac{\det (B-z)}{\det (A-z)}.\end{equation}
Moreover, one has 
$$E(\mu)=\langle A\mathbf v,\mathbf v \rangle$$
and 
$$\sigma_2(\mu)=\|PA(I-P)\|=\|(I-P)AP\|.$$
With respect to the orthogonal decomposition $\C^n=\{\mathbf v\}\oplus \mathbf v^\perp$, 
the operator $A$ takes the form 
\begin{equation}
\label{w-2}
A=\(\begin{array}{cc}E(\mu)& \mathbf d^*\\\mathbf c & B\end{array}\),
\end{equation}
where $\mathbf c,\mathbf d\in \mathbf v^\perp$ and $\|\mathbf c\|=\|\mathbf d\|=\sigma_2(\mu)$. 
By analogy with the invertibility criterion for diagonally dominant matrices we are led to the
following statement: \\

\begin{Clm} \label{weighted-conj} \label{wei} If $|E(\mu)|>\sigma_2(\mu)$  and all eigenvalues $w$ of $B$ satisfy 
$|w| >\sigma_2(\mu)$, then $A$ is invertible. \end{Clm}

  If the above statement were true, then the following statement about the zero location
of Cauchy transforms (\ref{w-1}) would also be true: 

\begin{Clm}\label{Cauchy-transform} \label{Cau} Let $S(\mu)=\{z_1,\dots,z_n\}$ and let $W_e(\mu)$ denote the set which is
the union of zeros of the Cauchy transform (\ref{w-1}) and the point $E(\mu)$. Then 
$$H(S(\mu),W_e(\mu))\le \sigma_2(\mu). $$\end{Clm}

We will see shortly that the above Claim is false for $n>2$. A natural attempt to save it would be to restrict 
both sets appearing in the inequality  to some relevant subsets. First we adopt some general notation and restrict $W_e$.
 For a compact subset $K\subset \C$, let $\sigma_\infty(K)$ denote the Chebyshev radius of $K$, i.e. 
the radius of the smallest closed disk containing $K$. If $\mu$ is a probability measure, we denote
by $S(\mu)$ the closed support of $\mu$ and we put $\sigma_\infty(\mu) = \sigma_\infty(S(\mu))$, 
and $W_e(\mu) = \{ z \in C; \ {\mathcal C}_\mu (z) = 0, \ \ z \notin S(\mu)\} \cup \{ E(\mu)\}.$

\begin{Clm} \label{varconj} $H(S(\mu),W_e(\mu))\le \sigma_\infty(\mu). $ \end{Clm}

A three-point mass example shows that Claim (\ref{varconj}) fails and hence so do Claims (\ref{wei} ) and (\ref{Cau} ).
Namely, the polynomial $F(z) = (z-1)(z^2+1)^n$ has the derivative $F'(z) = (z^2+1)^{n-1}[ (2n+1)z^2 -2nz + 1]$
 and the associated measure $\mu = \frac{1}{2n+1}\delta_1 + \frac{n}{2n+1} \delta_i +  \frac{n}{2n+1} \delta_{-i}.$
 Thus, $\sigma_\infty (\mu) = 1$, $S(\mu) = \{ 1,i,-i\}$ and $W_e(\mu) = \{ \frac{n}{2n+1} \pm \frac{\sqrt{n^2-2n-1}}{2n+1} \} \cup \{ \frac{1}{2n+1}\}.$
 If $n \geq 3$, all points of $W_e(\mu)$ are real and therefore the distance from $\pm i$ to $W_e(\mu)$ is strictly greater than $1 = \sigma_\infty (\mu)$.
This example also disproves a conjecture raised in \cite{SaffT}.

Finally, we restrict $S(\mu)$ by defining for a finite point mass measure $\mu$:
$$ S_{\min} (\mu) =\{ \zeta \in S(\mu); \ \ \mu(\{ \zeta \}) \leq \mu (\{w\}), \ \ w \in S(\mu)\}.$$
We are led to formulate another statement which still implies Sendov's conjecture.
This time $V(\mu) = \{ z \in C; \ {\mathcal C}_\mu (z) = 0, \ \ z \notin S(\mu)\}.$

\begin{Conj} \label{minvar} For a finite point mass probability measure $\mu$,
$$h(S_{\min}(\mu), V(\mu)) \leq \sigma_\infty(\mu).$$ \end{Conj}

Nattapong Bosuwan of Vanderbilt University has shown this conjecture holds in the case when the set $S_{min}(\mu)$ lies
on the circle $\vert z-c\vert =\sigma_{\infty}(\mu)$, where $c$ is the Chebyshev center of $S(\mu)$ as well as in the cases when $S(\mu )$ consists either of three point masses lying anywhere on the complex plane or an arbitrary number of collinear point masses. His proof of the first case is a generalization of the proof from \cite{GRR} that Sendov's conjecture holds for zeros on the unit circle.  We present Bosuwan's proof for the case where $S(\mu )$ consists of three point masses.

\begin{prop}\cite{Bos} If $\mu=\sum_{k=1}^{3}m_k \delta_{z_k},$  then Conjecture \ref{minvar} is true \emph{(}in fact with
$\sigma_{\infty}(\mu)$ replaced by $\sigma_{2}(\mu)$\emph{)}.\end{prop}

\begin{proof} Without loss of generality we assume that $m_1=1,$ and $m_2, m_3 \geq 1$ and
proceed as in the proof of Theorem \ref{3main}.  Define $v_{1}=\left(\frac{\sqrt{m_{1}}}{(w_{1}-z_{1})},\frac{\sqrt{m_{2}}}{(w_{1}-z_{2})},\frac{\sqrt{m_{3}}}{(w_{1}-z_{3})}\right)$, $v_{2}=\left(\frac{\sqrt{m_{1}}}{\overline{(w_{2}-z_{1})}},\frac{\sqrt{m_{2}}}{\overline{(w_{2}-z_{2})}},\frac{\sqrt{m_{3}}}{\overline{(w_{2}-z_{3})}}\right)$ and $v_{3}=(\sqrt{m_{1}},\sqrt{m_{2}},\sqrt{m_{3}})$ where $w_{1}, w_{2}$ are the elements of $V(\mu)$.

$\{v_{1},v_{2},v_{3}\}$ are mutually orthogonal and $\{\hat{v_{1}},{\hat{v_{2}},{\hat{v_{3}}}}\}$ is an orthonomal basis of $\mathbb{C}^{3}$ where $\hat{v_{i}}=\frac{v_{i}}{\left\|v_{i}\right\|}$.
Then \[U= \left( \begin{array}{ccc}
\hat{v_{1}} \\
\hat{v_{2}}\\
\hat{v_{3}} \end{array} \right)\] is an unitary matrix.
By Lemma 3.1, $\phi(U)$ is doubly stochastic.
Let $G:=\phi(U)$.
Then $g_{ij}=\frac{m_{j}{\left|w_{i}-z_{j}\right|}^{-2}}{\sum_{\alpha=1}^{3}m_{\alpha}\left|w_{i}-z_{\alpha}\right|^{-2}}$ , $i=1,2$ and $g_{3j}=\frac{m_{j}}{L}$ where $L=m_{1}+m_{2}+m_{3}$.
Since G is doubly stochastic, $\sum_{i=1}^{2}g_{i1}=\frac{L-1}{L}$.
This implies that there exists $i_{0}\in\{1,2\}$ such that $g_{i_{0}1}\geq\frac{1}{L}$.
Let $E=E(\mu)$.

Then from equation \ref{GL} with $n$ replaced by $L$ and
$i$ replaced by $i_0$  we get  $\left|w_{i_{0}}-z_{1}\right|\leq\sigma_{2}(\mu)\leq1$.
\end{proof}

  The Chebyshev radius $\sigma_\infty(\mu)$ admits an operator-theoretical interpretation. Namely, 
it was proved by Bj\"orck and Thom\'ee \cite{BT} that for any bounded normal operator 
$A$ in a Hilbert space $H$ 
$$\sigma_\infty(\sigma(A))=\inf_{\alpha\in\C}\|A-\alpha I\|= 
\sup_{v\in H\atop \|v\|\le 1}\(\|Av\|^2-|\langle Av,v \rangle|^2\)^{1/2}. $$
If $A$ is the operator of multiplication by an independent variable in $L^2(\mu)$ for a compactly supported probability measure $\mu$ we have 
$$\sigma_\infty(\mu)=\sup_{v\in L^2(\mu)\atop \|v\|\le 1}
\(\|Av\|^2-|\langle Av,v \rangle |^2\)^{1/2} = \sup\{\sigma_2(\nu)\,:\,\nu\ll \mu \}. $$

Since we are dealing with matrix interpretations, an alternative possibility is to define the zeros of the Cauchy transform (\ref{w-1})
as a spectrum is to use rank one perturbations of the above operator 
$A: \,(x_l)_{1\le l\le n}\mapsto (z_lx_l)_{1\le l\le n}. $ Given $a\in \C$, we introduce the rank-one  
operator $T_a$ as 
$$T_a: \, (x_l)_{1\le l\le n}\mapsto (z_l-a)_{1\le l\le n}\cdot\sum_{k=1}^n\alpha_kx_k. $$
One can check that the spectrum of the operator $A+T_a$ is the set 
$\{a,w_1,\dots,w_{n-1}\}$, where $w_1,\dots,w_{n-1}$ are zeros of (\ref{w-1}). In particular, 
if $a=E(\mu)$, then $\sigma(A+T_{E(\mu)})=W_e(\mu)$. Moreover, in this special case we have in 
addition the properties
$T_{E(\mu)}^2=0$ and $\|T_{E(\mu)}\|=\sigma_2(\mu)$. 

Sadly, in the labyrinth of observations we have stated and partially disproved in this section, we do not yet know how to state a weighted variance conjecture for probability measures supported by a continuum rather than finitely many points.

\section{Concluding remarks}

\subsection{An indefinite inner product condition which would imply 
the 2-variance conjecture}

As was noticed above, the main difficulty in proving Conjecture \ref{matrix} consists in transforming
known spectral properties of operators $B_\ell$ to metric properties. 
As an alternative, one can search for those metric properties of $B_\ell$ which would imply 
desired spectral properties. We describe one possible scenario.

As usual, we make first the reduction to the case where $z_\ell=0$ and $\sigma_2^2=1-\eps$.  Further, 
it is known from matrix inertia theory (see, e.g. \cite[chapter 13]{LT}) that the following 
condition is sufficient for the existence of an eigenvalue of an operator $B$ (acting in some Hilbert space $H$)
inside the unit disk: 
there exists an indefinite inner product $\langle \cdot,\cdot\rangle$ in $H$ such that 

 $\bullet$ $\langle B\mathbf x,B\mathbf x\rangle \le \langle \mathbf x,\mathbf x\rangle $ 
for any $\mathbf x\in H$; 

 $\bullet$ $\langle \mathbf x_0,\mathbf x_0\rangle >0$ for some $\mathbf x_0\in H$. 

Therefore, one way to attack the variance conjecture by operator theoretical methods 
might be trying to construct such an inner product.

\subsection{Maxwell's conjecture}  In connection with the discussion of the weighted variance conjecture in Section 9, it seems natural to raise here a question of plausible extensions to higher dimensions. Note that the rational function (\ref{w-1}), or more generally, the Cauchy transform
 $$C_{\mu}(z): = \int_ \C \frac{d\mu(\zeta)}{\zeta - z},$$
 where $\mu$ is any compactly supported probability measure in $\C$, can be interpreted as the force field (i.e., the complex gradient $\frac{\partial}{\partial z}$) of the logarithmic potential
 $$ L_{\mu} (z): = \int_\C \log|z-\zeta|^2 d\mu(\zeta).$$
 
 In this context, the Gauss-Lucas Theorem (for probability measures $\mu$) is obvious. The points where the gradient of the potential $L_{\mu}$ vanishes, i.e., the zeros of the Cauchy transform $C_\mu$ of the measure $\mu$, all lie in the convex hull of the support of $\mu$. (At any point $z_0$ outside of the convex hull there would be a non-trivial component of the force field $C_\mu$ ``pushing'' a probational positive charge placed at $z_0$ off to infinity.)
 
 In this form, the result extends word for word to higher dimensions if we replace the logarithmic potential by the Coulomb (or, in the context of gravitational force, Newtonian) potential in $\R^n$ thus replacing  $L_{\mu}$ by
 $$ N_{\mu}(x):= \int_{\R^n} \frac{d\mu(y)}{|x-y|^{n-2}} $$
 for $n\geq 3$. Some further generalizations of the Gauss-Lucas theorem to higher dimensions can be found in \cite{Good}, also cf. \cite{DBS}.  However, it turns out that for atomic measures $\mu$ even the estimate for the total number of critical points where the gradient of the potential vanishes, i.e., no force is present, is not known. A conjecture going back to Maxwell \cite{Maxwell} asserts that if the measure $\mu$ consists of $N$ point charges (not necessarily all positive), the total number of \textit{isolated} critical points of the potential does not exceed $(N-1)^2$. Of course, for $C_{\mu}$, which is a rational function in $\C$ of degree $N$ with a zero of order $N$ at $\infty$, the total number of finite critical points equals $(2N-2)-(N-1)=N-1$ as follows at once from the Fundamental Theorem of Algebra.
 
 Note that in $\R^n, n\geq 3$, the critical points of the potential can form lines or curves as well, e.g., if the $+1, -1$ charges alternate at the vertices of a square, the line through the center of the square which is perpendicular to the plane of the square will consist entirely of critical points of the potential. Examples of bounded curves, like circles and ellipses, of critical points in the electrostatic fields generated by three charges can be found in \cite{J}, also see \cite{Killian} which is more easily accessible for an English speaking reader.
 
 Maxwell simply asserted the claim without giving any justification for its validity. Only recently, in \cite{GNS}, in a technically brilliant paper, the authors were able to prove that the number of isolated critical points of the potential is finite. Yet, even for $3$ (!) charges their separate proof for the possible number of isolated critical points yields the estimate $\leq 12$ instead of $\leq 4$ asserted by Maxwell. A slight improvement of the estimate in \cite{GNS} under an additional assumption that all $n$ charges (in $\R^3$) lie in the same plane can be found in \cite{Killian}. Also, we refer the reader to \cite{GNS} for the history of the problem and a rather complete, although surprisingly short, list of references that exist up to this date.
 
We formulate an analog of Conjecture \ref{minvar} in higher dimensions.   

\begin{Conj} Let $\mu$ be a probability measure supported on a finite subset of $\mathbb{R}^{n}$. Then $$h(S_{\min}(\mu), V(\mu)) \leq \sigma_\infty(\mu),$$ where $V(\mu)$ denotes the set of critical points of the Newtonian potential. \end{Conj}

 \subsection{Miller's local maxima}

Let $n\geq 2$, $p\in [1,\infty]$ and let $S_{p}(n)$ denote the set of all monic $n$th degree polynomials $F$ for which $\sigma_{p}(F)\le 1$.  Since $h(F,F^{\prime})$ is continuous as a function of $F$ on the compact set $S_{p}(n)$, $h(F,F^{\prime})$ has an absolute maximum on $S_{p}(n)$.   Conjecture 1 is equivalent to the statement that the polynomials $F(z)=(z-c)^{n}-\omega$ where $\vert \omega \vert=1$  are absolute maxima for $h(F,F^{\prime})$ on $S_{p}(n)$.  Conjecture 2 is equivalent to the statement that for $p>1$, these polynomials are the only absolute maxima for $h(F,F^{\prime})$ on $S_{p}(n)$. 

In \cite{Mi}, Miller looked at the local maxima of $h(F,F^{\prime})$ on $S_{\infty}(n)$.  It was known that all polynomials of the form $F(z)=(z-c)^{n}-\omega$ where $\vert \omega \vert=1$ are local maxima of $h(F,F^{\prime})$ on $S_{\infty}(n)$ \cite{Mi2, VZ}.  Miller found eight properties which together formed a sufficent condition for a polynomial to be a local maximum of $h(F,F^{\prime})$ on $S_{\infty}(n)$.  He then found polynomials which possess all eight  properties; these polynomials were local maxima but were not of the form 
$(z-c)^{n}-\omega$ \cite{Mi}.  These unexpected local maxima are particularly useful as test cases and possible counterexamples for conjectures. The polynomial which disproved the conjecture in Section 3 that every column of a Gauss-Lucas matrix of $n$th degree polynomial has an element greater than or equal to $\frac{1}{n}$ was $F(z)=z^{19}-0.881444934z^{18}+0.896690269z^{17}-0.492806889$ which is a local maximum of $h(F,F^{\prime})$ on $S_{\infty}(19)$ found by Miller in \cite{Mi}.  It would be very instructive and useful to study the local maxima of $h(F,F^{\prime})$ on $S_{p}(n)$ for finite $p$ of which little is known.

\end{document}